\documentclass[11pt]{article}

\usepackage[margin=1in]{geometry}
\usepackage{amsmath,amssymb,amsthm}
\usepackage{mathtools}
\usepackage{enumitem}
\usepackage{hyperref}
\usepackage{xcolor}

\newif\ifdraftremarks
\draftremarksfalse   

\newtheorem{theorem}{Theorem}
\newtheorem{lemma}[theorem]{Lemma}
\newtheorem{proposition}[theorem]{Proposition}
\newtheorem{corollary}[theorem]{Corollary}
\theoremstyle{definition}

\newtheorem{remark}[theorem]{Remark}

\newcommand{\Fpm}{\{-1,+1\}}
\newcommand{\jc}{j}
\newcommand{\Jc}{J}

\title{A divisibility theorem for odd $J$-characteristics of two-level designs}

\author{Pieter Thijs Eendebak\thanks{Corresponding author: \texttt{pieter.eendebak@gmail.com}}}
\date{}

\begin{document}
\maketitle

\begin{abstract}
We prove a divisibility theorem for the signed $J$-characteristics of two-level
designs: if the number of factors $n$ is odd and every $J$-characteristic of a
proper odd-cardinality subset of factors vanishes, then the top
$J$-characteristic is divisible by $2^{n-1}$. As an arithmetic consequence, any two-level design whose
$J$-characteristics vanish in orders one, two, three, five, and seven but
which has a nonzero odd-order $J$-characteristic must have at least $256$ runs. This
settles, uniformly in the number of factors, a conjecture of Eendebak, Schoen,
Vazquez, and Goos (2023) on the nonexistence of certain strength-three
even--odd designs with $56$ or $64$ runs. The divisibility bound is sharp at
every odd order and is attained by the even-weight half-fraction.
\end{abstract}

\medskip
\noindent\textbf{Keywords:} two-level designs; 
$J$-characteristics; Walsh--Hadamard transform;
 even--odd designs;
strength-$3$ designs.

\smallskip
\noindent\textbf{MSC 2020:} 05B15 (primary); 62K15, 06E30 (secondary).

\section{Introduction}

Write $[n]:=\{1,2,\dots,n\}$ for the index set of factors. A two-level design
with $N$ runs on $n$ factors is a multiset of $N$ \emph{runs}
$r=(r_1,\dots,r_n)\in\Fpm^n$, where $r_i\in\{-1,+1\}$ is the level of the
$i$-th factor in run~$r$. Its (signed) \emph{$J$-characteristics}%
~\cite{DengTang1999,Tang2001} are the integer sums
\begin{equation}\label{eq:jchar-intro}
  \jc(s)=\sum_{\text{runs }r}\prod_{i\in s} r_i, \qquad s\subseteq[n],
\end{equation}
indexed by subsets of factors. They are the integer coordinates of the design
in the Walsh--Hadamard basis~\cite{ODonnell2014}. We write $\Jc(s)=|\jc(s)|$
and $J_q\equiv 0$ when $\jc(s)=0$ for every $s\subseteq[n]$ of cardinality $q$.
A design has \emph{strength~$t$} if $J_1\equiv\dots\equiv J_t\equiv 0$.
Throughout we write divisibility as $a\mid b$, read ``$a$ divides $b$,''
meaning $b=ak$ for some integer $k$.

This paper is motivated by an application to \emph{even--odd} designs of
strength~$3$: strength-$3$ designs which carry both a nonzero even-order and a
nonzero odd-order $J$-characteristic. Eendebak, Schoen, Vazquez and
Goos~\cite{ESVG2023} conjectured, on the basis of partial enumeration, that
no such design exists at $N\in\{56,64\}$ once $n\ge 9$ and $J_5\equiv J_7\equiv
0$. We prove a sharper and structurally transparent statement that implies the
conjecture uniformly in~$n$.

\paragraph{Main results.}

\begin{enumerate}[nosep]
  \item (Theorem~\ref{thm:divis}) If $n$ is \emph{odd} and every odd-order
        characteristic except possibly the top one vanishes, then
        $2^{\,n-1}\mid \jc([n])$.
  \item (Corollary~\ref{cor:main}) Any strength-3 design with $J_5\equiv
        J_7\equiv 0$ that has any nonzero odd-order $J$-characteristic must
        satisfy $N\ge 2^{\,8}=256$. Consequently no such design exists for
        $N<256$, for any $n$.
\end{enumerate}

\section{Preliminaries}

A two-level design with $N$ runs is
encoded by its row-multiplicity vector $f:\Fpm^n\to\mathbb Z_{\ge 0}$ with
$\sum_x f(x)=N$. For $s\subseteq[n]$ set
\begin{equation}\label{eq:chi-jc}
  \chi_s(x)=\prod_{i\in s} x_i\in\Fpm,
  \qquad
  \jc(s)=\sum_{x\in\Fpm^n} f(x)\,\chi_s(x).
\end{equation}
Each $\jc(s)$ is an integer with $|\jc(s)|\le N$. When several designs appear simultaneously we
write $\jc_f(s)$ to indicate the underlying function; otherwise the subscript
is omitted. Lemmas~\ref{lem:orth} and~\ref{lem:inv} below are standard Walsh--Hadamard identities (see e.g.~\cite{ODonnell2014}, \S1.2); we include self-contained proofs for completeness.

\begin{lemma}[Orthogonality]\label{lem:orth}
For $x,y\in\Fpm^n$,
\begin{equation}\label{eq:orth}
  \sum_{s\subseteq[n]} \chi_s(x)\,\chi_s(y)
  =\begin{cases} 2^n & x=y,\\ 0 & x\ne y.\end{cases}
\end{equation}
\end{lemma}

\begin{proof}
$\chi_s(x)\chi_s(y)=\prod_{i\in s} x_i y_i$, so
$\sum_{s\subseteq[n]}\chi_s(x)\chi_s(y)=\sum_{s\subseteq[n]}\prod_{i\in s}x_iy_i=\prod_{i=1}^n(1+x_iy_i)$,
the last step by expanding the product over all subsets. Each factor is $2$ when
$x_i=y_i$ and $0$ otherwise.
\end{proof}

\begin{lemma}[Inversion]\label{lem:inv}
For every $f:\Fpm^n\to\mathbb Z$ and every $x\in\Fpm^n$,
\begin{equation}\label{eq:inv}
  2^n\, f(x)=\sum_{s\subseteq[n]} \jc(s)\,\chi_s(x).
\end{equation}
\end{lemma}

\begin{proof}
Expand $\jc(s)=\sum_y f(y)\chi_s(y)$, swap sums, and apply
Lemma~\ref{lem:orth}:
\[
\sum_s\jc(s)\chi_s(x)
=\sum_y f(y)\sum_s\chi_s(y)\chi_s(x)
=\sum_y f(y)\cdot 2^n[y{=}x]
=2^n f(x).\qedhere
\]
\end{proof}

\section{The divisibility theorem}

\begin{theorem}[Divisibility of the top odd characteristic]\label{thm:divis}
Let $n\ge 1$ be \emph{odd} and let $f:\Fpm^n\to\mathbb Z$. If
\begin{equation}\label{eq:divhyp}
  \jc(s)=0 \quad\text{for every }s\subseteq[n]\text{ with $|s|$ odd and }s\ne[n],
\end{equation}
then $2^{\,n-1}\mid \jc([n])$.
\end{theorem}

\begin{proof}
Apply Lemma~\ref{lem:inv} at $x$ and at $-x$. Since
$\chi_s(-x)=(-1)^{|s|}\chi_s(x)$, subtracting gives
\[
  2^n\bigl(f(x)-f(-x)\bigr)
   = \sum_{s\subseteq[n]} \jc(s)\bigl(1-(-1)^{|s|}\bigr)\chi_s(x)
   = 2\!\!\sum_{|s|\text{ odd}}\!\!\jc(s)\,\chi_s(x),
\]
because $1-(-1)^{|s|}$ is $0$ for even $|s|$ and $2$ for odd $|s|$. By
\eqref{eq:divhyp} the only odd-cardinality $s$ with $\jc(s)\ne 0$ is
$s=[n]$ --- here we use that $n$ is odd, so $[n]$ itself has odd cardinality.
The right-hand sum collapses to $\jc([n])\chi_{[n]}(x)$, and dividing by~$2$:
\begin{equation}\label{eq:key}
  2^{\,n-1}\bigl(f(x)-f(-x)\bigr)=\jc([n])\,\chi_{[n]}(x).
\end{equation}
Take $x=(+1,\dots,+1)$, so $\chi_{[n]}(x)=1$; the left-hand side is
$2^{\,n-1}$ times the integer $f(x)-f(-x)$, hence
$2^{\,n-1}\mid \jc([n])$.
\end{proof}

\begin{corollary}[Run-size bound]\label{cor:runbound}
Under the hypothesis of Theorem~\ref{thm:divis}, if $\jc([n])\ne 0$ then
$|\jc([n])|\ge 2^{\,n-1}$. If moreover $f\ge 0$, so that $f$ is the
row-multiplicity vector of a design, then $N=\sum_x f(x) \ge 2^{\,n-1}$.
\end{corollary}

\begin{proof}
$2^{\,n-1}\mid \jc([n])$ and $\jc([n])\ne 0$ give $|\jc([n])|\ge 2^{\,n-1}$,
and for $f\ge 0$ we have $|\jc(s)|\le\sum_x f(x)=N$ for every~$s$.
\end{proof}

\section{Projection and the main corollary}

For $s\subseteq[n]$ the \emph{projection} $\pi_s f$ of $f$ onto the columns
indexed by~$s$ is the function on $\Fpm^{s}$ obtained by summing over all
coordinates outside~$s$:
\begin{equation}\label{eq:projdef}
  (\pi_s f)(y)=\sum_{\substack{x\in\Fpm^n\\x_s=y}} f(x),\qquad y\in\Fpm^{s},
\end{equation}
where $x_s$ denotes the restriction of $x$ to the coordinates in~$s$.

\begin{lemma}[Projection preserves sub-characteristics]\label{lem:proj}
For every $s\subseteq[n]$ and every $u\subseteq s$,
$\jc_{\pi_s f}(u)=\jc_f(u)$.
\end{lemma}

\begin{proof}
Since $u\subseteq s$, $\chi_u(x)$ depends only on $x_s$. Substituting the
definition of~$\pi_s f$:
\[
\jc_{\pi_s f}(u)
   =\sum_{y\in\Fpm^s}(\pi_s f)(y)\,\chi_u(y)
   =\sum_{x\in\Fpm^n} f(x)\,\chi_u(x_s)
   =\sum_{x\in\Fpm^n} f(x)\,\chi_u(x)=\jc_f(u).\qedhere
\]
\end{proof}

\begin{corollary}[Main corollary; resolves the ESVG conjecture]\label{cor:main}
Let $A$ be a two-level strength-$3$ design with $N$ runs satisfying
$J_5\equiv J_7\equiv 0$. If $A$ has any nonzero odd-order
$J$-characteristic, then
\begin{equation}\label{eq:mainbound}
  N\ \ge\ 2^{\,q-1}\ \ge\ 2^{\,8}=256,
\end{equation}
where $q$ is the smallest odd order with a nonzero characteristic.
Consequently, for every $N<256$ --- and in particular for $N\in\{56,64\}$ ---
no such design exists, for any number of factors.
\end{corollary}

\begin{proof}
Let $f$ be the row-multiplicity vector of~$A$.
Let $q$ be the least odd integer with some $q$-subset $s$ satisfying
$\jc_f(s)\ne 0$; such $q$ exists by hypothesis. Strength~$3$ kills orders
$1$ and~$3$ and the hypotheses kill orders $5$ and~$7$, so $q\ge 9$.

By minimality of~$q$, every proper odd subset $u\subsetneq s$ has
$\jc_f(u)=0$. Project onto~$s$: by Lemma~\ref{lem:proj}, $\pi_s f$ is a
$q$-factor design with $\jc_{\pi_s f}(u)=\jc_f(u)$ for every $u\subseteq s$.
Hence in $\pi_s f$ every odd-order characteristic vanishes except possibly
the top one $\jc_{\pi_s f}(s)=\jc_f(s)\ne 0$. Theorem~\ref{thm:divis} applied to the $q$-factor design $\pi_s f$ gives
$2^{\,q-1}\mid \jc_f(s)$, so by Corollary~\ref{cor:runbound}
$N\ge 2^{\,q-1}\ge 2^{\,8}=256$.
\end{proof}

\section{Sharpness}

The threshold of Corollary~\ref{cor:runbound} is attained.

\begin{proposition}\label{prop:sharp}
For every $q\ge 1$ let $H_q=\{x\in\Fpm^q:\chi_{[q]}(x)=+1\}$ be the
even-weight half-fraction, with multiplicity one. Then $H_q$ has $2^{\,q-1}$
runs, $\jc([q])=2^{\,q-1}$, and $\jc(s)=0$ for every
$\varnothing\ne s\ne[q]$.
\end{proposition}

\begin{proof}
$H_q$ is the kernel of the group homomorphism $x\mapsto\chi_{[q]}(x)$ from
$(\Fpm^q,\cdot)$ to $\Fpm$, hence has $2^{\,q-1}$ elements. For $s=[q]$,
$\chi_{[q]}\equiv 1$ on~$H_q$, so $\jc([q])=2^{\,q-1}$. For
$\varnothing\ne s\ne[q]$, $\chi_s$ restricts to a nontrivial character of
the group~$H_q$, with sum~$0$; it is nontrivial because, picking $i\in s$
and $k\notin s$, the vector with $-1$ in exactly positions $i$ and~$k$ has
even weight, hence lies in~$H_q$, and has $\chi_s=-1$.
\end{proof}

For $q=7$ the design $H_7$ is the regular $2^{7-1}_{\mathrm{VII}}$ fractional
factorial~\cite{BoxHunter1961,BoxHunterHunter2005}: a strength-$6$ design at $N=64$ with $J_5\equiv 0$ and $J_7=64$.
Hence the hypothesis $J_7\equiv 0$ in Corollary~\ref{cor:main} --- equivalently
the restriction $n\ge 9$ in the original conjecture --- cannot be dropped
at~$N=64$.

\section*{Acknowledgements}

We thank Eric Schoen for comments that improved the exposition. The proofs and
exposition in this paper were developed in interactive sessions with
Anthropic's Claude (Opus~4.7), prompted to formulate, prove, and refine the
divisibility theorem and its application to the ESVG conjecture; the human
author selected the problem, supplied feedback at each iteration, and verified
the final argument.

\end{document}